\documentclass[a4paper,11pt]{amsart}

\usepackage{epsfig}
\usepackage{amsthm,amsfonts}
\usepackage{amssymb,graphicx,color}
\usepackage[all]{xy}
\usepackage{verbatim}
\usepackage{hyperref}
\usepackage{enumitem}

\newtheorem{theorem}{Theorem}[section]
\newtheorem*{theorem*}{Theorem}

\newtheorem{corollary}[theorem]{Corollary}

\newtheorem{definition}[theorem]{Definition}
\newtheorem{example}[theorem]{Example}
\newtheorem{remark}[theorem]{Remark}
\newtheorem{conjecture}{Conjecture}
\newtheorem{question}{Question}
\newtheorem{claim}{Claim}

\newtheorem{innercustomthm}{Theorem}
\newenvironment{customthm}[1]
  {\renewcommand\theinnercustomthm{#1}\innercustomthm}
  {\endinnercustomthm}

\newcommand{\C}{\mathbb C}
\newcommand{\R}{\mathbb R}

\newcommand{\N}{\mathbb N}

\begin{document}

\title[On Lipschitz Geometry at infinity of complex analytic sets]
{On Lipschitz Geometry at infinity of complex analytic sets}

\author[J. E. Sampaio]{Jos\'e Edson Sampaio}

\address{Jos\'e Edson Sampaio: Departamento de Matem\'atica, Universidade Federal do Cear\'a,
	      Rua Campus do Pici, s/n, Bloco 914, Pici, 60440-900, 
	      Fortaleza-CE, Brazil. \newline  
              E-mail: {\tt edsonsampaio@mat.ufc.br}
}


\thanks{
% The first named author was partially supported by CNPq-Brazil grant .
% The second named author was partially supported by CNPq-Brazil grant .
The author was partially supported by CNPq-Brazil grant 303811/2018-8.
}

\begin{abstract}
In this article, we study the Lipschitz Geometry at infinity of complex analytic sets and we obtain results on algebraicity of analytic sets and on Bernstein's problem.
Moser's Bernstein Theorem says that a minimal hypersurface which is a graph of an entire Lipschitz function must be a hyperplane. H. B. Lawson, Jr. and R. Osserman presented examples showing that an analogous result for arbitrary codimension is not true. In this article, we prove a complex non-parametric version of Moser's Bernstein Theorem. More precisely, we prove that any entire complex analytic set in $\mathbb{C}^n$ which is Lipschitz regular at infinity must be an affine linear subspace of $\mathbb{C}^n$. In particular, a complex analytic set which is a graph of an entire Lipschitz function must be affine linear subspace. That result comes as a consequence of the following characterization of algebraic sets, which is also proved here: if $X$ and $Y$ are entire complex analytic sets which are bi-Lipschitz homeomorphic at infinity then $X$ is a complex algebraic set if and only if $Y$ is a complex algebraic set too. Thus, an entire complex analytic set is a complex algebraic set if and only if it is bi-Lipschitz homeomorphic at infinity to a complex algebraic set. No restrictions on the singular set, dimension nor codimension are required in the results proved here. 
\end{abstract}

\maketitle

\section{Introduction}

Looking to scrutinize global Lipschitz Geometry of complex analytic sets in some sense, we arrived in the following definition (which was recently studied in \cite{BobadillaFS:2018}, \cite{FernandesS:2020}, \cite{Sampaio:2019} and \cite{Targino:2019}):

\begin{definition}
Let $X\subset \R^n$ and $Y\subset\R^m$ be two subsets. We say that $X$ and $Y$ are {\bf bi-Lipschitz homeomorphic at infinity}, if there exist compact subsets $K\subset\R^n$ and $\widetilde K\subset \R^m$ and a bi-Lipschitz homeomorphism $\phi \colon X\setminus K\rightarrow Y\setminus \widetilde K$. In this case, we also say that $X$ and $Y$ are {\bf bi-Lipschitz equivalent at infinity}.
\end{definition}
Let us recall the definitions of Lipschitz and bi-Lipschitz mappings.
\begin{definition}
Let $X\subset\R^n$ and $Y\subset\R^m$. A mapping $f\colon X\rightarrow Y$ is called {\bf Lipschitz} if there exists $\lambda >0$ such that is
$$\|f(x_1)-f(x_2)\|\le \lambda \|x_1-x_2\|$$ for all $x_1,x_2\in X$. A Lipschitz mapping $f\colon X\rightarrow Y$ is called
{\bf bi-Lipschitz} if its inverse mapping exists and is Lipschitz.
\end{definition}

The equivalence class of a set $X$ in the above equivalence relation is called the {\bf Lipschitz Geometry at infinity of $X$}. 
In this article, we present some implications of the study on Lipschitz Geometry at infinity of complex analytic sets to the studies on algebraicity and rigidity of complex analytic sets.

% \begin{definition}
% We say that a set $X\subset \C^n$ is an {\bf entire complex analytic set} (or shortly {\bf entire set}) if there exists a family of entire functions $\mathcal{F}$ such that $X=\{x\in \C^n; f(x)=0$ for all $f\in\mathcal{F}\}$.
% \end{definition}

About algebraicity of complex analytic sets, maybe the most important result is the famous Chow's Theorem, which says that any closed complex analytic subset of a complex projective space is a complex algebraic subset. Chow's Theorem has many important consequences and versions and is the basis for applying analytic methods in Algebraic Geometry. An important Chow-type theorem, which is used here, is the result known by Stoll-Bishop's Theorem, which says that a entire complex analytic set (i.e. a set $X\subset \C^n$ such that there exists a family of entire functions $\mathcal{F}$ such that $X=\{x\in \C^n; f(x)=0$ for all $f\in\mathcal{F}\}$) is an algebraic set if and only if its volume growth is bounded (this and other versions of Chow's Theorem can be seen, for instance, in \cite{Bishop:1964}, \cite{Rudin:1967}, \cite{Stoll:1964a}, \cite{Stoll:1964b}, \cite{Tworzewski:1981} and the references cited therein). In order to be more precise, Stoll-Bishop's Theorem says the following:
\begin{theorem}[Stoll-Bishop's Theorem (see \cite{Stoll:1964a,Stoll:1964b,Bishop:1964})]\label{Stoll-Bishop}
Let $Z\subset\C^k$ be a pure $d$-dimensional entire complex analytic subset. Then $Z$ is a complex algebraic set if and only if there exists a constant $R>0$ such that 
$$
\frac{\mathcal{H}^{2d}(Z\cap \overline{B}_r^{2k}(0))}{r^{2d}}\leq R, \quad \mbox{ for all } r>0,
$$
where $\mathcal{H}^{2d}(Z\cap \overline{B}_r^{2k}(0))$ denotes the $2d$-dimensional Hausdorff measure of $Z\cap \overline{B}_r^{2k}(0)=\{x\in Z;\|x\|\leq r\}$.
\end{theorem}

In this article, we prove the following Chow-type theorem that relates algebraicity of analytic sets with Lipschitz Geometry at infinity of complex analytic sets. 

\begin{customthm}{\ref*{chow-type-thm}}
Let $X\subset\C^n$ be a pure $d$-dimensional entire complex analytic subset. Then the following statements are mutually equivalent:
\begin{enumerate}
 \item [(1)] $X$ is a complex algebraic set.
 \item [(2)] $X$ has a unique tangent cone at infinity which is a $d$-dimensional complex algebraic set;
%  \item [(3)] $\Theta_{X}(\infty)< +\infty$;
 \item [(3)] $X$ is bi-Lipschitz homeomorphic at infinity to a complex algebraic set.
\end{enumerate}
\end{customthm}

We are going to address the notion of tangent cone at infinity in Subsection \ref{cones}.

Another aim of this article is to show that Lipschitz Geometry at infinity of complex analytic sets is also related with the famous Bernstein's problem. In order to know, the celebrated theorem due to J. Moser in \cite{Moser:1961}, called Moser's Bernstein Theorem, says that a minimal hypersurface which is a graph of an entire Lipschitz function must be a hyperplane. So, it is natural to  ask  whether  Moser's theorem  can  be  generalized to the higher codimensional case.   In other words, it is natural to ask the following:
\begin{question}\label{question-moser-gen}
Given a minimal variety $M\subset \R^{n+m}$ which is the graph of Lipschitz mapping $f\colon\R^n\to\R^m$, is $M$ an affine linear subspace?
\end{question}

In general, Question \ref{question-moser-gen} has a negative answer. Lawson and Osserman in \cite{LawsonO:1977} presented a minimal cone which is the graph of a Lipschitz mapping, but is not an affine linear subspace, more precisely, they presented the following:
\begin{example}[Theorem 7.1 in \cite{LawsonO:1977}]\label{example:LO}
The graph of the Lipschitz mapping $f\colon \R^4\to \R^3$ given by
$$
f(x)=\frac{\sqrt{5}}{2}\|x\|\eta\left(\frac{x}{\|x\|}\right), \quad \forall x\not=0,
$$
is a minimal cone, where $\eta\colon \mathbb{S}^3\to \mathbb{S}^2$ is the Hopf mapping given by
$$
\eta(z_1,z_2)=(|z_1|^2-|z_2|^2,2z_1\bar z_2).
$$
\end{example}
However, several mathematicians gave positive answers to Question \ref{question-moser-gen} with additional conditions or approached some version of Question \ref{question-moser-gen}  (for example, see \cite{JostX:1999}, \cite{JostXY:2016}, \cite{JostXY:2018} and \cite{Xin:2005}).

Recently, in \cite{FernandesS:2020} was proved some Bernstein-type theorems in higher codimension for complex analytic sets. For example, it was proved that a complex algebraic set which is Lipschitz regular at infinity (see Definition \ref{def:Lip_regular}) must be an affine linear subset. % and it was also proved that an entire $d$-dimensional complex analytic set which is LNE at infinity (see Definition \ref{def:LNE_infinity}) and has a $d$-dimensional linear subspace as its unique tangent cone at infinity must be an affine linear subspace. 
So, we have the following natural question:

\begin{question}\label{question-regularity}
If $X$ is a pure dimensional entire complex analytic set which is Lipschitz regular at infinity then is $X$ an affine linear subspace?
\end{question}

In this article, we prove the following result, which answers positively Question \ref{question-regularity}.
\begin{customthm}{\ref*{main theorem}}
Let $Z\subset\C^n$ be a pure $d$-dimensional entire complex analytic set. If $Z$ is Lipschitz regular at infinity, then $Z$ is an affine linear subspace of $\C^n$.
\end{customthm}

We provide also the following Bernstein-type theorem in higher codimension:
\begin{customthm}{\ref*{gen_Moser_one}}
Let $X\subset\C^{d+n}$ be a $d$-dimensional entire complex analytic subset and $K\subset \C^d$ be a compact set such that $X\setminus (K\times \C^n)$ is the graph of a mapping $\varphi\colon \C^d\setminus K\to \C^n$. 
Assume that 
\begin{itemize}
 \item [(i)] $K=\emptyset$ when $d=1$,
\end{itemize}
and suppose that there exist a constant $C>0$, $k\in \mathbb{N}$ and a sequence $\{t_j\}_{j\in \N}$ of real positive numbers such that $t_j\to +\infty$ satisfying the following
\begin{itemize}
 \item [(ii)] $\|\varphi(x)\|\leq C t_j^k$, whenever $x\in \{u\in \C^d\setminus K; \|u\|= t_j\}$ and
 \item [(iii)] $\varphi|_{B_{t_j}(0)\setminus K}$ is a bounded mapping, for all $j\in \N$.
\end{itemize}
Then $X$ contains a $d$-dimensional complex algebraic set of $\C^{d+n}$ and $\varphi$ is the restriction of a polynomial mapping with degree at most $k$. Additionally, if $X$ is also an irreducible set then $X$ is a complex algebraic set with degree at most $k$ and, in particular, if $k=1$ then $X$ is an affine linear subspace of $\C^{d+n}$.

% Let $X\subset\C^{d+n}$ be a complex analytic subset and $K\subset \C^d$ be a compact set such that $X\setminus (K\times \C^n)$ is the graph of a mapping $\varphi\colon \C^d\setminus K\to \C^n$. 
% Suppose that there exist a constant $C>0$ and a sequence $\{t_j\}_{j\in \N}$ of real positive numbers such that $t_j\to +\infty$ satisfying the following
% \begin{itemize}
%  \item [(i)] $\|\varphi(x)\|\leq C t_j$, whenever $x\in \{(u,v)\in \C^{d+n}; \|u\|= t_j\}$ and
%  \item [(ii)] $\varphi|_{B_{t_j}(0)\setminus K}$ is a bounded mapping, for all $j\in \N$.
% \end{itemize}
% If $d>1$ or $K=\emptyset$ then $X$ is an affine linear subspace of $\C^{d+n}$ and, in particular, $\varphi$ is the restriction of an affine mapping.
\end{customthm}
In particular, we obtain a positive answer to Question \ref{question-moser-gen} when $M$ is a complex analytic set (see Corollary \ref{gen_Moser_two}). %, which is a complex version of the result proved by Moser in \cite{Moser:1961}.
% \begin{customcor}{\ref*{gen_Moser_two}}
% Let $X\subset\C^{n+d}$ be an entire complex analytic subset. If $X$ is a graph of a Lipschitz mapping $\varphi\colon \C^d\to \C^n$, then $X$ is an affine linear subspace of $\C^{n+d}$.
% \end{customcor}

% a proof that any complex analytic set in $\mathbb{C}^n$ which is Lipschitz regular at infinity must be an affine linear subspace of $\mathbb{C}^n$. In particular, a complex analytic subset of $\C^{n+d}$ which is a graph of a Lipschitz entire mapping $\varphi\colon\C^k\to \C^n$ is an affine linear subspace of $\C^n$. 
% The main ingredients of these proofs stand on the notions of tangent cone at  infinity and Lipschitz regularity at infinity.

Notice that, we address complex analytic sets which are not necessarily graph of smooth functions; a priori, they are not supposed  even smooth and in Theorem \ref{gen_Moser_one}, 
we do not impose that $\varphi$ is even a continuous mapping.
Let us remark also that Theorem \ref{gen_Moser_one} does not hold true in general (see Remark \ref{rem:gen_Moser_one_sharp}), if we remove one of the conditions (i)-(iii), and Example \ref{example:LO} shows that Theorem \ref{gen_Moser_one} does not hold true for closed sets which minimizes volume (stable varieties).

\section{Preliminaries}\label{preliminaries}
All the subsets of $\R^n$ (or $\C^n$) are considered equipped with the induced Euclidean metric.

\subsection{Lipschitz regularity at infinity}
\begin{definition}\label{def:Lip_regular}
A subset $X\subset\R^n$ is called {\bf Lipschitz regular at infinity} if $X$ and $\R^k$ are bi-Lipschitz homeomorphic at infinity, for some $k\in\N$.
\end{definition}

\begin{example}\label{real_example} Let $X\subset\R^{n+1}$ $(n>0)$ be defined by $ X = \{ (x_1,...,x_n,x_{n+1})\in\R^{n+1}; \ x^2+...+x_n^2=x_{n+1}^3 \}.$ We see that $X$ is a real algebraic subset of $\R^{n+1}$ with an isolated singularity at $0\in\R^{n+1}.$ By using the mapping $\pi\colon X\rightarrow\R^n$; $\pi (x_1,...,x_n,x_{n+1})=(x_1,...,x_n)$, it is easy to see that $X$ is Lipschitz regular at infinity.
\end{example}

\begin{example} Let $Y\subset\C^2$ be defined by $ Y = \{ (x,y)\in\C^2; \ x^2=y \}.$ We see that $Y$ is a smooth algebraic subset of $\C^2$. From another way,  $Y$ is not Lipschitz regular at infinity. 
\end{example}

% 
% \begin{definition}
% Let $Z\subset\C^k$ be a pure $d$-dimensional entire complex analytic subset. Thus, we define the function $n_Z\colon (0,+\infty)\to \R$ given by
% $$
% n_Z(r):=\frac{\mathcal{H}^{2d}(Z\cap \overline{B}_r^{2k}(0))}{r^{2d}},
% $$
% where $\mathcal{H}^{2d}(Z\cap \overline{B}_r^{2k}(0))$ denotes the $2d$-dimensional Hausdorff measure of $Z\cap \overline{B}_r^{2d}(0)$.% and $\mu_{2d}$ is the volume of the $2d$-dimensional unit ball.
% \end{definition}

% 
% \subsection{Inner distance}\label{inner distance}
% Given a path connected subset $X\subset\R^m$ the
% \emph{inner distance}  on $X$  is defined as follows: given two points $x_1,x_2\in X$, $d_X(x_1,x_2)$  is the infimum of the lengths of paths on $X$ connecting $x_1$ to $x_2$. As we said in the beginning of Section \ref{preliminaries} all the sets considered in this paper are supposed to be equipped with the Euclidean induced metric. Whenever we consider the inner distance, we emphasize it clearly.
% 
% \begin{definition}[See \cite{BirbrairM:2000}] A subset $X\subset\R^n$ is called {\bf Lipschitz normally embedded at infinity} if there exist a compact subset $K\subset \R^n$ and $\lambda >0$ such that
% $$d_X(x_1,x_2)\le \lambda \|x_1-x_2\|$$
% for all $x_1,x_2\in X\setminus K$.
% \end{definition}
% 
% 
% \begin{corollary}\label{neighborhood}
% Let $X\subset\C^n$ be a complex algebraic subset. If $X$ is Lipschitz regular at infinity, then there exists a compact subset $K\subset\C^n$ such that $X\setminus K$ is Lipschitz normally embedded.
% \end{corollary}

\subsection{Tangent cone at infinity}\label{cones}

\begin{definition}
Let $X\subset \R^m$ be a subset such that $p\in \overline{X}$. Given a sequence of real positive numbers $\{ t_j \}_{j\in \N}$ such that $t_j\to +\infty$, we say that $v\in \R^m$ is {\bf tangent to $X$ at infinity with respect to $\{t_j\}_{j\in \N}$} if there is a sequence of points $\{x_j\}_{j\in \N}\subset X$ such that $\lim\limits _{j\to +\infty }\frac{1}{t_j}x_j=v$.
\end{definition}

\begin{definition}\label{def:unique_tg_cone}
Let $X\subset \R^m$ be a subset such that $p\in \overline{X}$ and $T=\{t_j\}_{j\in \N}$ be a sequence of real positive numbers $\{ t_j \}_{j\in \N}$ such that $t_j\to +\infty$. We denote the set of all vectors which are tangents to X at infinity w.r.t. $T$ by $C^T_{\infty}(X)$. When $C^T_{\infty}(X)$ and $C^S_{\infty}(X)$ coincide for any other sequence of real positive numbers $S=\{ s_j \}_{j\in \N}$ such that $s_j\to +\infty$, we denote $C^T_{\infty}(X)$ by $C_{\infty}(X)$ and we call it the {\bf tangent cone of $X$ at infinity}.
\end{definition}

Let us remark that a tangent cone of a set $X$ at infinity can be non-unique as we can see in the following example.
\begin{example}
Let $X=\{(x,y)\in\R^2;\, y\cdot\sin ({\rm log} (x^2+y^2+1))=0\}$. For each $j\in \N$, we define $t_j=(e^{j\pi}-1)^{\frac{1}{2}}$ and $s_j=(e^{j\pi+\pi/2}-1)^{\frac{1}{2}}$. Thus, for $T=\{t_j\}_{j\in \N}$ and $S=\{s_j\}_{j\in \N}$, we have $(0,1)\in C^T_\infty(X)\setminus C^S_\infty(X)$ and, thus, $C^T_\infty(X))\not=C^S_\infty(X)$.
\end{example}
See in \cite{SampaioS:2021} some characterizations of sets with unique tangent cones.

As it was already said in \cite{FernandesS:2020}, in general, it is not an easy task to verify whether unbounded subsets have a unique tangent cone at infinity, even in the case of some classes of analytic subsets, for instance, concerning to such a problem, there is a still unsettled conjecture by Meeks III (\cite{Meeks:2005}, Conjecture 3.15) stating that: {\it any properly immersed minimal surface in $\R^3$ of quadratic area growth has a unique tangent cone at infinity}.

% Let us finish this Section pointing out the following results which we are going to use in the proof of Theorems \ref{chow-type-thm} and \ref{main theorem}.
% 
% \begin{lemma}[Corollaries 2.16 and 2.18 in \cite{FernandesS:2020}]\label{dim_cone}
% Let $Z \subset \R^n$ be an unbounded semialgebraic set. Then $Z$ has a unique tangent cone at infinity, $C_{\infty }(Z)$ is a semialgebraic set and $\dim_{\R} C_{\infty }(Z)\leq \dim_{\R} Z$.
% \end{lemma}
% 
% \begin{lemma}[Theorem \cite{FernandesS:2020}]\label{tg_cones}
% Let $A\subset \R^m$ and $B\subset\R^n$ be unbounded semialgebraic subsets. If $A$ and $B$ are bi-Lipschitz homeomorphic at infinity, then their tangent cones at infinity $C_{\infty }(A)$ and $C_{\infty }(B)$ are bi-Lipschitz homeomorphic.
% \end{lemma}

\section{Algebraicity and Lipschitz Geometry at infinity: A result like Chow's Theorem}

% So, we have the following characterization of complex algebraic sets:
\begin{theorem}\label{chow-type-thm}
Let $X\subset\C^n$ be a pure $d$-dimensional entire complex analytic subset. Then the following statements are mutually equivalent:
\begin{enumerate}
 \item [(1)] $X$ is a complex algebraic set.
 \item [(2)] $X$ has a unique tangent cone at infinity which is a $d$-dimensional complex algebraic set;
%  \item [(3)] $\Theta_{X}(\infty)< +\infty$;
 \item [(3)] $X$ is bi-Lipschitz homeomorphic at infinity to a complex algebraic set.
\end{enumerate}
\end{theorem}
\begin{proof}
Our proof is organized in the following way: we are going to proof that (1) is equivalent to (2), and (1) is equivalent to (3).

\noindent $(1) \Rightarrow(2)$.
By using a global \L ojasiewicz inequality, it was already proved in \cite{LeP:2018} that any complex algebraic set has a unique tangent cone at infinity, which is a complex algebraic set. Here, we present a different proof, without using any \L ojasiewicz inequality.

% For a polynomial $f\colon \C^n\to \C$ of degree $m$, we write
% $$f=f_0+f_{1}+\cdots+f_{m-1}+f_m$$ where each $f_k$ is a homogeneous polynomial of degree $k$ and we define $f^*:=f_m$ and $\tilde f\colon \C^{n+1}\to \C$ by 
% $$
% \tilde f(x_0,x)=\sum\limits_{k=0}^m x_0^{m-k}f_k(x). %=x_0^df(\frac{1}{x_0}x)
% $$

Let $\pi\colon \C^{n+1}\to \C P^n$ be the canonical projection.
Assume that $X$ is a complex algebraic set. By Corollary 2.16 in \cite{FernandesS:2020}, $X$ has a unique tangent cone at infinity. If $X$ is a bounded set, then $C_{\infty}(X)=\{0\}$, which is an algebraic set. Hence, we can assume that $X$ is an unbounded set. Thus, the closure of $X$ in $\C P^n$, denoted by $\overline{X}$, is a complex algebraic subset of $\C P^n$ (see \cite[Section 7.2, Proposition 1]{Chirka:1989}) and $\widetilde{X}=\pi^{-1}(\overline{X})\cup \{0\}$ is a homogeneous complex algebraic set of $\C^{n+1}$ (see \cite[Section 7.1, proof of Theorem]{Chirka:1989}). 
Let $\tilde p_1,...,\tilde p_r\colon \C^{n+1}\to \C$ be homogeneous polynomials such that $\widetilde{X}=\{(x_0,x)\in \C^{n+1};\tilde p_1(x_0,x)=...=\tilde p_r(x_0,x)=0\}$. In particular, $Z=\overline{X} \cap H_{\infty}$ is a complex algebraic subset of $\C P^n$ satisfying $Z=\{(x_0:x)\in \C P^n; x_0=0$ and $\tilde p_1(x_0,x)=...=\tilde p_r(x_0,x)=0\}$, where $H_{\infty}$ is the hyperplane at infinity, i.e., $H_{\infty}=\{(x_0:x_1:....:x_n)\in \C P^n; x_0=0\}$. For each $i\in \{1,...,r\}$, let $p_i\colon \C^{n}\to \C$ be the polynomial given by $p_i(x)=\tilde p_i(0,x)$ for all $x\in \C^n$ and, thus, we define $C^*(X)=\{x\in \C^{n};p_1(x)=...=p_r(x)=0\}$. In particular, $\widetilde{X}\cap (\{0\}\times \C^n)=\{0\}\times C^*(X)$.

% satisfying $\overline{X}=\{(x_0:x)\in \C P^n;\tilde f(x_0,x)=0$ for all $f\in \mathcal{I}(X)\}$, where $\mathcal{I}(X)$ denotes the ideal of all polynomials in $\C[x_1,...,x_n]$ which vanish on $X$ (see, for example, \cite[Section 7.2, Proposition 1]{Chirka:1989}). In particular, $Z=\overline{X} \cap H_{\infty}$ is a complex algebraic subset of $\C P^n$ satisfying $Z=\{(x_0:x)\in \C P^n; x_0=0$ and $f^*(x)=0$ for all $f\in \mathcal{I}(X)\}$, where $H_{\infty}$ is the hyperplane at infinity, i.e., $H_{\infty}=\{(z_0:z_1:....:z_n)\in \C P^n; z_0=0\}$. Let $V^*(\mathcal{I}(X))=\{x\in \C^n; f^*(x)=0$ for all $f\in \mathcal{I}(X)\}$. % By doing the obvious identification $H_{\infty}\equiv \C P^{n-1}$, 
% Then, there exist homogeneous polynomials $f_1,...,f_k\colon \C^n\to \C$ such that $Z=\{(0:z)\in H_{\infty}; f_1(z)=...=f_k(z)=0\}$. 

\begin{claim}\label{algebraic_cone}
$C_{\infty}(X)=C^*(X)$.
\end{claim}
\begin{proof}[Proof of Claim \ref{algebraic_cone}]  
Since $C_{\infty}(X)$ and $C^*(X)$ are real cones, i.e., if $v\in C_{\infty}(X)$ (resp. $v\in C^*(X)$) then $\lambda v\in C_{\infty}(X)$ (resp. $\lambda v\in C^*(X)$) for all $t\geq 0$, it is enough to show that $C_{\infty}(X)\cap \mathbb{S}^{2n-1}=C^*(X)\cap \mathbb{S}^{2n-1}$. 
Since $C_{\infty}(X_1\cup X_2)=C_{\infty}(X_1)\cup C_{\infty}(X_2)$ and $C^*(X_1\cup X_2)=C^*(X_1)\cup C^*(X_2)$, we assume that $X$ is an irreducible algebraic set. 

Let $v\in C_{\infty}(X)$ such that $\|v\|=1$. Then there exists a sequence $\{z_j\}_j\subset X$ such that $\lim\limits_{j\to +\infty} z_j=\infty$ and $\lim\limits_{j\to +\infty} \frac{z_j}{\|z_j\|}=v$. Thus, for each $j$, let $L_j=\pi^{-1}(1:z_j)\cup \{0\}$ and $w_j\in L_j\cap \mathbb{S}^{2n+1}$. Since $(\frac{1}{\|z_j\|}:\frac{z_j}{\|z_j\|})\to (0:v)$ as $j\to +\infty$, we obtain, taking a subsequence, if necessary, that $w_j\to (0,\lambda v)$ for some $\lambda \in \C$ with $|\lambda|=1$. Since $\widetilde{X}$ is a homogeneous complex algebraic set, $(0,v)\in \widetilde{X}$ and, thus, $v\in C^*(X)$.
This shows that $C_{\infty}(X)\subset C^*(X)$.

% $$
% 0=\lim\limits_{j\to +\infty}\frac{f(z_j)}{\|z_j\|^m}= \lim\limits_{j\to +\infty}\frac{f^*(z_j)}{\|z_j\|^m}=f^*(v).
% $$
% This shows that $C_{\infty}(X)\subset V$.

Let $v\in C^*(X) $ be a vector such that $\|v\|=1$. Since $\{0\}\times C^*(X)=\pi^{-1}(Z)\cup \{0\}$, there exists a sequence $\{(x_{0j},y_j)\}_j\subset  \widetilde{X} \setminus (\{0\}\times C^*(X))$ which converges to $(0,v)$. We can assume that $x_{0j}>0$ for all $j$. By Curve Selection Lemma (see \cite[Theorem 3.13]{Coste:2000}), there exists a semialgebraic arc $\gamma\colon [0,\varepsilon)\to \widetilde{X}$ such that $\gamma(0)=(0,v)$ and $\gamma((0,\varepsilon))\subset \widetilde{X}\cap ((0,+\infty)\times \C^n)$. We can parametrize $\gamma$ such that $\gamma(t)=(t,\beta(t))$, where $\beta$ satisfies $\lim\limits_{t\to 0^+}\beta(t)=v$. Therefore, $\frac{\beta(t)}{t}\in X$ for all $t\in (0,\varepsilon)$. Let $\{t_j\}_j\subset (0,\varepsilon)$ be a sequence which converges to $0$. Then, the sequence $\{z_j\}_j\subset X$ given by $z_j=\frac{\beta(t_j)}{t_j}$ satisfies $\lim\frac{z_j}{\|z_j\|}=v$, which shows that $v\in C_{\infty}(X)$.
Therefore, $C^*(X)\subset C_{\infty}(X)$, which finishes the proof of Claim \ref{algebraic_cone}.
% a sequence $\{(x_{0j}:y_j)\}_j\subset  \overline{X} \setminus H_{\infty}$ which converges to $(0:v)$. By multiplying $(x_{0j}:y_j)$ by $\frac{\|x_{0j}\|}{x_{0j}}$, if necessary, we assume that $x_{0j}$ is a positive real number for all $j$. Therefore, $z_j=\frac{1}{x_{0j}}y_j\in X$ for all $j$ and, taking subsequence, if necessary, we can assume that $\{\frac{z_j}{\|z_j\|}\}_j$ converges and, in this case, $\lim \frac{z_{j}}{\|z_{j}\|}=v$, which shows that $v\in C_{\infty}(X)$.

% Since $(0:\lambda v)\in \{(0:z)\in H_{\infty}; f^*(x)=0$ for all $f\in \mathcal{I}(X)\}$, there also exists a sequence $\{(\tilde x_{0j}:\tilde y_j)\}_j\subset  \overline{X} \setminus H_{\infty}$ which converges to $(0:\lambda v)$. 
\end{proof}
Since $C^*(X)$ is a homogeneous complex algebraic set, we obtain that $C_{\infty}(X)$ is a homogeneous complex algebraic set as well.
Since $\dim_{\C} Z\geq \dim_{\C} \overline{X} + \dim_{\C}H_{\infty}-n\geq d+n-1-n=d-1$, then $\dim_{\C}C_{\infty}(X)\geq d$. By Corollary  2.18 in \cite{FernandesS:2020}, $\dim_{\R}C_{\infty}(X)\leq \dim_{\R}X=2d$. Therefore $C_{\infty}(X)$ is a $d$-dimensional complex algebraic set.

\bigskip

\noindent $(2) \Rightarrow (1)$. Assume that $X$ has a unique tangent cone at infinity which is a $d$-dimensional complex algebraic set $Y=C_{\infty}(X)$. 
% By Wirtinger's Theorem.... 
Let $\{e_1,...,e_n\}$ be the canonical basis of $\C^n$. For $\lambda\in \Lambda:=\{\lambda\colon \{1,...,k\}\to \{1,...,n\}; \lambda(1)<...<\lambda(k)\}$, we write $\C_{\lambda}=span\{e_{\lambda(1)},...,e_{\lambda(k)}\}$ and $\pi_{\lambda}\colon \C^n\to \C_{\lambda}$ is the orthogonal projection.

By making a linear change of coordinates, if necessary, we may assume that $\pi_{\lambda}^{-1}(0)\cap Y=\{0\}$ for all $\lambda\in \Lambda$ (see, for example, \cite[Chapter 17, p. 226, Corollary]{Chirka:1989}). In particular, $\pi_{\lambda}|_X\colon X\to \C_{\lambda}$ is a proper mapping. Thus, there exists a complex analytic set $\sigma\subset \C_{\lambda}$ such that $\dim \sigma <d$ and $f_{\lambda}:=\pi|_{X\setminus \pi^{-1}(\sigma)}\colon X\setminus \pi_{\lambda}^{-1}(\sigma)\to \C_{\lambda}\setminus \sigma$ is a covering mapping with $k_{\lambda}$ fibers, for some $k_{\lambda}\in \mathbb{N}$. Since $f_{\lambda}$ is a Lipschitz mapping, by coarea formula (see \cite[Lemma 5]{Tworzewski:1981}), we obtain
\begin{eqnarray*}
\mathcal{H}^{2d}(X\cap B_r(0))&=&\sum\limits_{\lambda\in\Lambda(n,d)}\int_{\C_{\lambda}}\mathcal{H}^{0}(\pi_{\lambda}^{-1}(y)\cap X\cap B_r(0))d\mathcal{H}^{2d}(y)\\
       &\leq& \sum\limits_{\lambda\in\Lambda(n,d)}k_{\lambda}\mathcal{H}^{2d}(\C_{\lambda}\cap B_r(0))\\
       &\leq& k\binom{n}{d}\mu_{2d}r^{2d},      
\end{eqnarray*}
for all $r>0$, where $\mu_{2d}$ is the volume of the $2d$-dimensional unit ball and $k=\max \{k_{\lambda};\lambda \in\Lambda(n,d)\}$. By Stoll-Bishop's Theorem (Theorem \ref{Stoll-Bishop}), $X$ is a complex algebraic set.

\bigskip

\noindent $(1) \Rightarrow (3)$.
Since the identity mapping is a bi-Lipschitz mapping, we obtain that (1) implies (3). 

\bigskip

\noindent $(3) \Rightarrow (1)$. We assume that $X$ is bi-Lipschitz homeomorphic at infinity to a complex algebraic set. Let $A\subset \C^m$ be a complex algebraic set such that there exists compact subsets $K\subset \C^n$ and $\tilde K\subset \C^m$ and a bi-Lipschitz homeomorphism $\varphi\colon A\setminus \tilde K\to X\setminus  K$. Let $\lambda\geq 1$ such that
$$
\frac{1}{\lambda}\|x-y\|\leq \|\varphi(x)-\varphi(y)\|\leq \lambda \|x-y\|, \quad  \forall x,y\in A\setminus \tilde K.
$$
% Since $K$ and $\tilde K$ are compact subsets, there  exists $R\geq 1$

Fix $x_0\in A\setminus \tilde K$ and set $y_0=\varphi(x_0)$. Let $\tilde r_0=\|x_0\|$ and $r_0=\|y_0\|$. Thus, for any $y\in X\setminus K$ with $\|y\|= r$, let $x\in A\setminus \tilde K$ such that $y=\varphi(x)$, then we have
\begin{eqnarray*}
\|x\|&\leq& \|\varphi^{-1}(y)-\varphi^{-1}(y_0)\|+\|\varphi^{-1}(y_0)\|\\
              &\leq&\lambda\|y-y_0\|+ \|x_0\|\\
              &\leq&\lambda(r+r_0)+r_0              
\end{eqnarray*}
% and for any $x\in A\setminus \tilde K$ with $\|x\|=\tilde r$, we have
% \begin{eqnarray*}
% \|\varphi(x)\|&\leq& \|\varphi(x)-\varphi(x_0)\|+\|\varphi(x_0)\|\\
%               &\leq&\lambda\|x-x_0\|+ r_0\\
%               &\leq&\lambda(\tilde r+\tilde r_0)+r_0              
% \end{eqnarray*}
% and 
Thus, for any $r>0$, we obtain that 
$$
(X\setminus K)\cap B_{r}^{2n}(0)\subset \varphi ((A\setminus \tilde K)\cap B_{\lambda(r+r_0)+r_0}^{2m}(0)).
$$

Therefore, for any $r>0$, we have
\begin{eqnarray*}
\mathcal{H}^{2d}(X\cap B_{r}^{2n}(0))&= &\mathcal{H}^{2d}((X\setminus K)\cap B_{r}^{2n}(0))+\mathcal{H}^{2d}(X\cap K\cap B_{r}^{2n}(0))\\
    &\leq &\mathcal{H}^{2d}(\varphi ((A\setminus \tilde K)\cap B_{\lambda(r+r_0)+r_0}^{2m}(0)))+ \mathcal{H}^{2d}(X\cap K)\\
    &\leq &\lambda^{2d}\mathcal{H}^{2d}((A\setminus \tilde K)\cap B_{\lambda(r+r_0)+r_0}^{2m}(0))+ \mathcal{H}^{2d}(X\cap K)\\
    &\leq &\lambda^{2d}\mathcal{H}^{2d}(A\cap B_{\lambda(r+r_0)+r_0}^{2m}(0))+ \mathcal{H}^{2d}(X\cap K).
\end{eqnarray*}

Since $A$ is a complex algebraic set, by Stoll-Bishop's Theorem, there exists a constant $C>0$ such that $\mathcal{H}^{2d}(A\cap B_{t}^{2m}(0))\leq C t^{2d}$, for all $t>0$. Since $\mathcal{H}^{2d}(X\cap K)<+\infty$, there exists $R_1>0$ such that
$$
\frac{\mathcal{H}^{2d}(X\cap B_{r}^{2n}(0))}{r^{2d}}\leq \frac{\lambda^{2d}C(\lambda(r+r_0)+r_0)^{2d}+\mathcal{H}^{2d}(X\cap K)}{r^{2d}}\leq R_1,
$$
for all $r\geq 1$. But $X$ is analytic at $0$, then 
$$
\lim\limits_{r\to 0^+}\frac{\mathcal{H}^{2d}(X\cap B_{r}^{2n}(0))}{r^{2d}}=\mu_{2d}m(X,0),
$$
where $m(X,0)$ denotes the multiplicity of $X$ at $0$ (see \cite[Theorem 7.3]{Draper:1969}). Thus, there exists $R_2>0$ such that 
$$
\frac{\mathcal{H}^{2d}(X\cap B_{r}^{2n}(0))}{r^{2d}} \leq R_2,
$$
for all $r\leq 1$.

Therefore, by Stoll-Bishop's Theorem, $X$ is a complex algebraic set.
% 
% \begin{claim}
% $X$ is a complex algebraic set if and only if $\Theta_{X}(\infty)< +\infty$.
% \end{claim}
% 
% By the same reason, $Y$ is a complex algebraic set if and only if $\Theta_{Y}(\infty)< +\infty$.
% 
% Our finishes if one proves the following claim.
% \begin{claim}
% $\Theta_{X}(\infty)< +\infty$ if and only if $\Theta_{Y}(\infty)< +\infty$.
% \end{claim}
\end{proof}

\begin{remark}
After this paper was finished, we learned that just some days before, it was proved in \cite{DiasR:2021} that the items (1) and (2) of Theorem \ref{chow-type-thm} are equivalent. However, we notice that our proof is independent and this is why we keep it here.
\end{remark}

\begin{corollary}
Let $X\subset\C^n$ and $Y\subset \C^m$ be two entire complex analytic subsets. Assume that $X$ and $Y$ are bi-Lipschitz homeomorphic at infinity. Then $X$ is a complex algebraic subset if and only if $Y$ is a complex algebraic subset as well.
\end{corollary}

\begin{corollary}
Let $A\subset\C P^n$ be a subset. Assume that there exists a hyperplane $H\subset\C P^n$ such that $A\setminus H$ is a pure $d$-dimensional complex analytic subset of $\C P^n\setminus H$. Then the following statements are equivalent: 
\begin{itemize}
 \item [(1)] $A$ is a $d$-dimensional complex algebraic subset of $\C P^n$;
 \item [(2)] $A\cap H$ is a $(d-1)$-dimensional complex algebraic subset of $\C P^n$;
 \item [(3)] $A\cap H$ is a $(d-1)$-dimensional complex analytic subset of $H$;
\end{itemize}
\end{corollary}

\section{Some results like Moser's Bernstein Theorem}\label{main}
% At this moment we are ready to state the main result of the paper, which is a result like the Theorem J in (\cite{GreeneW:1979}, p. 180) (see also \cite{SiuY:1977} and \cite{Itoh:1979}).
\subsection{Lipschitz regularity at infinity: Parametric version}

\begin{theorem}\label{main theorem}
Let $Z\subset\C^n$ be a pure $d$-dimensional entire complex analytic subset. If $Z$ is Lipschitz regular at infinity, then $Z$ is an affine linear subspace of $\C^n$.
\end{theorem}
Before starting the proof of Theorem \ref{main theorem}, we do a slight digression to remind the notion of inner distance on a connected Euclidean subset.

Let $E\subset\R^{m}$ be a path-connected subset. Given two points $p,q\in E$, we define the {\bf inner distance} in $E$ between $p$ and $q$ by the number $d_E(p,q)$ below:
$$d_E(p,q):=\inf\{ \mbox{length}(\gamma) \ | \ \gamma \ \mbox{is an arc on} \ E \ \mbox{connecting} \ p \ \mbox{to} \ q\}.$$
\begin{proof}[Proof of Theorem \ref{main theorem}]
% Theorem \ref{main theorem} follows from Theorem \ref{chow-type-thm} jointly with Theorem 3.8 in \cite{FernandesS:2020}. However, for convenience of the reader, we present a proof here. 

Let $K_1\subset\C^n$ and $ K_2\subset \C^d$ be compact subsets such that there exists a bi-Lipschitz homeomorphism $\phi\colon Z\setminus  K_1\to \C^d\setminus K_2$. 
Without loss of generality, one can suppose that $K_2$ is the Euclidean closed ball $\overline {B_r(0)}$, for some $r>0$. 
% Let us denote $A=(X\setminus  K_1)\times \{0\}$ and $B=\{0\}\times (Y\setminus   K_2)$. 
By taking $\C^N=\C^n\times\C^d$ and $X=Z\times \{0\}\subset \C^N$, there exists a bi-Lipschitz map $\varphi\colon \C^N\to\C^N$ such that $\varphi(A)=B$, where $A=X\setminus  (K_1\times \{0\})$ and $B=\{0\}\times (\C^d\setminus   K_2)$ (see Lemma 3.1 in \cite{Sampaio:2016}). Let $\lambda>0$ be a constant such that
\begin{equation*}
\frac{1}{\lambda}\|x-y\|\leq \|\varphi(x)-\varphi(y)\|\leq \lambda\|x-y\|, \quad \forall x,y\in \C^N.
\end{equation*}

For each $k\in\N$, we define the mappings $\varphi_k\colon \C^N\to \C^N$ given by $\varphi_k(v)=\frac{1}{k}\varphi(kv)$. Thus, each $\varphi_k$ satisfies the following
\begin{equation*}
\frac{1}{\lambda}\|x-y\|\leq \|\varphi_k(x)-\varphi_k(y)\|\leq \lambda\|x-y\|, \quad \forall x,y\in \C^N.
\end{equation*}
So, there exist a subsequence $T=\{t_{j}\}_{j\in \N}\subset \N$ and a bi-Lipschitz mapping $d\varphi\colon\C^N\to \C^N$ such that $\varphi_{t_{j}}\to d\varphi$ uniformly on compact subsets of $\C^N$ (see more details in \cite[Theorem 2.19]{FernandesS:2020}). Furthermore, 
$$
\frac{1}{\lambda}\|u-v\|\leq \|d\varphi(u)-d\varphi(v)\|\leq \lambda\|u-v\|, \quad \forall u,v\in \C^N
$$
and $d\varphi(C_{\infty}^T(A))=C_{\infty}^T(B)=\{0\}\times\C^d$. In particular, $C_{\infty}^T(A)$ is a topological manifold. By Theorem \ref{chow-type-thm}, $X$ is a complex algebraic set and has a unique tangent cone at infinity, which is a homogeneous complex algebraic set. Thus, $A$ has a unique tangent cone at infinity and $C_{\infty}^T(A)=C_{\infty}(A)=C_{\infty}(X)$. By Prill's Theorem (see \cite[Theorem]{Prill:1967}), which says that any complex cone (i.e. a union of one-dimensional linear subspaces of $\C^N$) that is a topological manifold must be a linear subspace of $C^N$, we obtain that $C_{\infty}(X)$ is a linear subspace of $\C^N$.

Let us choose linear coordinates $(x,y)$ in $\C^N$ such that $$E:=C_{\infty }(X)=\{(x,y)\in\C^N;\,y=0\}$$ 
and let $P\colon \C^N\rightarrow E$ be the orthogonal projection. 
% There exist positive constants $C$ and $\rho$ such that $X\subset \{(x,y);\|y\|< C\|x\|\}\cup B_{\rho}$.
% 
% If $\gamma \colon(\varepsilon,\infty )\rightarrow X$ is an arc such that $\lim\limits_{t\to +\infty}\|\gamma(t)\|=+\infty $ and $P\circ\gamma(t)=tv+o_{\infty}(t)$, then $\gamma(t)=tv+o_{\infty}(t)$. 
Since the linear subspace $E$ is the tangent cone of $X$  at infinity, we have that $P|_{ X}\colon X\rightarrow E$ is a proper mapping. In fact, it is easy to show that there exist a constant $R>0$ and a compact set $K\subset \C^N$ such that $\|y\|\leq C\|P(x,y)\|= C\|x\|$ for all $(x,y)\in X\setminus K$.  
Then $P|_{ X}\colon X\rightarrow C_{\infty }(X)$ is a ramified cover with degree $<+\infty$ (see \cite{Chirka:1989}, Corollary 1 in the page 126). This means that there exists a codimension $\geq 1$ complex algebraic subset $\Sigma$ of the linear space $E$, called the ramification locus of $P|_X$, such that $\pi:=P|_{X\setminus P^{-1}(\Sigma)}\colon X\setminus P^{-1}(\Sigma)\to E\setminus \Sigma$ is a cover mapping with degree ${\rm deg} (\pi)<+\infty$.

\begin{claim}\label{claim:degree}
${\rm deg} (\pi)=1$.
\end{claim}

Let us assume, for a moment, that Claim \ref{claim:degree} holds true. Then $P|_{X\setminus P^{-1}(\Sigma)}\colon X\setminus P^{-1}(\Sigma)\to E\setminus \Sigma$ is a bijective mapping. Let $F\colon E\setminus \Sigma \to X\setminus P^{-1}(\Sigma)$ be its inverse. Since $\dim \Sigma\leq d-1$, $F$ extends to a holomorphic mapping $\tilde F\colon E \to X$. Since $\|y\|\leq C\|P(x,y)\|= C\|x\|$ for all $(x,y)\in X\setminus K$ and $\tilde F$ is bounded in the compact set $P(K)$, then there exists a constant $\tilde R$ such that $\|\tilde F(x,0)\|\leq \tilde R(\|x\|+1)$ for all $(x,0)\in E$. By Liouville's Theorem, $\tilde F$ is an affine mapping. Since $X$ is a pure dimensional set, we obtain that $X$ is an affine linear subspace of $\C^N$, which implies that $Z$ is an affine linear subspace of $\C^n$.

Thus, in order to finish the proof, we only need to prove Claim \ref{claim:degree}.
\begin{proof}[Proof of Claim \ref{claim:degree}]
Since $\Sigma$ is a codimension $\geq 1$ complex algebraic subset $\Sigma$ of the linear space $E$, then there exists a unit tangent vector $v_0\in C_{\infty }(X)\setminus C_{\infty }(\Sigma)$.
Since $v_0$ is not tangent to $\Sigma$ at infinity, there exist positive real numbers $\eta$ and $R$ such that
$$C_{\eta,R}=\{v\in C_{\infty }(X);\ \|v-tv_0\|< \eta t, \mbox{ for some } t>R \}$$
does not intersect the set $\Sigma$. 

Let us suppose that the degree ${\rm deg} (X)$ is strictly greater than 1. Thus, we have at least two different liftings $\gamma_1(t)$ and $\gamma_2(t)$ of the half-line $r(t)=tv_0$,  i.e. $P(\gamma_1(t))=P(\gamma_2(t))=tv_0$.  

On the other hand, since $C_{\eta,R }$ is a simply connected set, any path in $X$ connecting $\gamma_1(t)$ to $\gamma_2(t)$ is the lifting of a loop, based at the point $tv_0$ which is not contained in $C_{\eta,R }$. Thus, the length of such a path must be at least  $2\eta t$. It implies that the inner distance satisfies $d_{A}(\gamma_1(t),\gamma_2(t))\geq 2\eta t$.

Since
$$ \frac{1}{\lambda}\|p-q\|\leq \|\psi(p)-\psi(q)\|\leq \lambda \|p-q\|, \, \forall p,q\in B,$$
where $\psi=\varphi^{-1}$, it follows that
$$\frac{1}{\lambda}d_B(p,q)\leq d_A(\psi(p),\psi(q))\leq \lambda d_B(p,q), \, \forall p,q\in B.$$
On the other hand,  $d_B(p,q)\leq \pi\|p-q\|$, for all $p,q\in B$ and, by those inequalities above, it follows that
$$d_A(\gamma_1(t),\gamma_2(t))\leq \lambda^2\pi\|\gamma_1(t)-\gamma_2(t)\|, \, \forall \ t\in (r,+\infty).$$

But, $P(\gamma_1(t))=P(\gamma_2(t))=tv_0$ and $P$ is the orthogonal projection on $C_{\infty }(X)$, then $\gamma_1(t)$ and $\gamma_2(t)$ are tangent at infinity, that is,
$$\frac{\|\gamma_1(t)-\gamma_2(t)\|}{t}\to 0 \ \mbox{ as } \ t\to +\infty, $$
hence:
\begin{eqnarray*}
2\eta &\leq & \frac{d_{X}(\gamma_1(t),\gamma_2(t))}{t} \\
&\leq &\lambda^2\pi \frac{\|\gamma_1(t)-\gamma_2(t)\|}{t}\to 0 \ \ \mbox{as} \ t\to +\infty,
\end{eqnarray*}
which is a contradiction, since $\eta$ is a positive constant, which finishes the proof.
\end{proof}
\end{proof}

Notice that Theorem \ref{main theorem} does not hold true (with same assumptions) for real algebraic sets (cf. Example \ref{real_example} above).

As a consequence, we obtain one of the main results of \cite{FernandesS:2020}.
% \begin{corollary}[Theorem 1.1 in \cite{FernandesS:2020}]
% Let $X\subset\C^n$ be a closed and pure $d$-dimensional analytic subset. Suppose $X$ has a unique tangent cone at infinity and this cone is a $d$-dimensional complex
% linear subspace of $\C^n$. If $X$ is Lipschitz Normally Embedded at infinity, then $X$ is an affine linear subspace of $\C^n$.
% \end{corollary}
% \begin{proof}
% 
% \end{proof}

\begin{corollary}[Theorem 3.8 in \cite{FernandesS:2020}]
Let $X\subset\C^n$ be a pure dimensional entire complex algebraic subset. If $X$ is Lipschitz regular at infinity, then $X$ is an affine linear subspace of $\C^n$.
\end{corollary}

Another consequence is the following Lipschitz version of Cancellation Law:
\begin{corollary}\label{weak-cancellation}
Let $X\subset\C^n$ be an entire complex analytic subset. If there exists $k\in \mathbb{N}$ such that $X\times \C^k$ is bi-Lipschitz homeomorphic to $\C^{d+k}$, then $X$ is a $d$-dimensional affine linear subspace of $\C^n$ and, in particular, $X$ is bi-Lipschitz homeomorphic to $\C^{d}$.
\end{corollary}
% 
% Let us a prove the following strong version of Corollary  \ref{weak-cancellation}.
% \begin{corollary}
% Let $X\subset\C^n$ and $Y\subset \C^m$ be entire pure dimensional complex analytic subsets. If there exists $k\in \mathbb{N}$ such that $X\times \C^k$ is bi-Lipschitz homeomorphic to $Y\times \C^k$ then $X$ is bi-Lipschitz homeomorphic to $Y$.
% \end{corollary}

We finish this Subsection with the following conjecture:

\begin{conjecture}
Let $X\subset\C^n$ be a pure $d$-dimensional entire complex analytic subset. Suppose that there exists a sequence $T=\{t_j\}_{j\in \N}$ of real positive numbers such that $t_j\to +\infty$ and $E=C_{\infty}^T(X)$ is a $d$-dimensional complex linear subspace of $\C^n$. If $X$ is LNE at infinity, then $X$ is an affine linear subspace of $\C^n$.
\end{conjecture}

\subsection{Lipschitz regularity at infinity: Non-parametric version}

% \begin{lemma}\label{Liouville_version}
% Let $\{t_j\}_{j\in\mathbb{N}}$ be a sequence of positive numbers such that $\lim t_j=+\infty$ and let $f\colon \mathbb{C}^m\to \mathbb{C}$ be a complex analytic function that satisfies $|f(x)|\leq C t_j^d$ for all $x\in B^m_{t_j}(0)$ and for all $j$ and some $C\geq 0$ and $d\in \mathbb{N}$. Then $f$ is a polynomial of degree at most $d$.
% \end{lemma}
% \begin{proof}
% For each $x\in \mathbb{C}^m$, let $j_0\in \mathbb{N}$ such that $x\in B^m_{t_j}(0)$ for all $j\geq j_0$. Since $f$ satisfies $|f(x)|\leq C t_j$ for all $j\geq j_0$, by Cauchy's estimate, for each $\alpha$ a multi-index, we have
% \begin{eqnarray*}
% \left|D^{\alpha} f(x)\right|&\leq &\frac{\alpha !}{(2\pi i)^m }  \frac{\|f\|_{L^{\infty}(B_{t_j-\|x\|}^m(x))}}{(t_j-\|x\|)^{|\alpha|}} \\
% &\leq &\frac{\alpha !}{(2\pi i)^m } \frac{\|f\|_{L^{\infty}(B_{t_j}^m(0))}}{(t_j-\|x\|)^{|\alpha|}} \\
% &\leq &\frac{\alpha !}{(2\pi i)^m }\frac{Ct_j^d}{(t_j-\|x\|)^{|\alpha|}} 
% \end{eqnarray*}
% for all $j\geq j_0$, which implies $\left|D^{\alpha} f(x)\right|=0$ for any $\alpha$ satisfying $|\alpha|>d$.
% Therefore $f$ is a polynomial of degree at most $d$.
% \end{proof}
% \begin{remark}
% In fact, the same is true for a harmonic function, i.e., if $\{t_j\}_{j\in\mathbb{N}}$ is a sequence of positive numbers such that $\lim t_j=+\infty$ and $u\colon \mathbb{R}^m\to \mathbb{R}$ is a harmonic function that satisfies $|u(x)|\leq C t_j^d$ for all $x\in B^m_{t_j}(0)$ and for all $j$ and some $C\geq 0$ and $d\in \mathbb{N}$. Then $u$ is a polynomial of degree at most $d$.
% \end{remark}

\begin{theorem}\label{gen_Moser_one}
Let $X\subset\C^{d+n}$ be a $d$-dimensional entire complex analytic subset and $K\subset \C^d$ be a compact set such that $X\setminus (K\times \C^n)$ is the graph of a mapping $\varphi\colon \C^d\setminus K\to \C^n$. 
Assume that 
\begin{itemize}
 \item [(i)] $K=\emptyset$ when $d=1$,
\end{itemize}
and suppose that there exist a constant $C>0$, $k\in \mathbb{N}$ and a sequence $\{t_j\}_{j\in \N}$ of real positive numbers such that $t_j\to +\infty$ satisfying the following
\begin{itemize}
 \item [(ii)] $\|\varphi(x)\|\leq C t_j^k$, whenever $x\in \{u\in \C^d\setminus K; \|u\|= t_j\}$ and
 \item [(iii)] $\varphi|_{B_{t_j}(0)\setminus K}$ is a bounded mapping, for all $j\in \N$.
\end{itemize}
Then $X$ contains a $d$-dimensional complex algebraic set of $\C^{d+n}$ and $\varphi$ is the restriction of a polynomial mapping with degree at most $k$. Additionally, if $X$ is also an irreducible set then $X$ is a complex algebraic set with degree at most $k$ and, in particular, if $k=1$ then $X$ is an affine linear subspace of $\C^{d+n}$.
\end{theorem}
\begin{remark}
In Theorem \ref{gen_Moser_one}, 
we do not impose that $\varphi$ is a continuous mapping and, in particular, 
$X$ is not necessarily graph of smooth functions; a priori, it is not supposed even smooth. Moreover, no restrictions on the singular set, dimension nor codimension are required. 
\end{remark}

\begin{proof}[Proof of Theorem \ref{gen_Moser_one}]
By taking subsequence, if necessary, we may assume that $K \subset B_{t_1}(0)$ and $t_j< t_{j+1}$ for all $j$.
We consider the orthogonal projection $P\colon \C^n\rightarrow \C^d.$ For each $i\in\{1,...,2n\}$, we consider the coordinate function $\phi_i\colon X\to \R$ given by $\phi_i(r_1,...,r_{2n})=r_i$. 

\begin{claim}\label{inequality}
$X \cap \{(x,y)\in \C^{d+n}; t_1\leq \|x\|\leq t_j\}\subset X \cap \{(x,y)\in \C^{d+n}; \|y\|\leq 2nC t_j^k\}$  for all $j$.
\end{claim}
\begin{proof}[Proof of Claim \ref{inequality}]
Indeed, if Claim \ref{inequality} is not true, there exist $j>1$ and $(x_0,y_0)\in X$ such that $t_1\leq \|x_0\|\leq t_j$ and $\|y_0\|> 2nC t_j^k$. Therefore, there exists a coordinate function $\phi_i:X\to \R$ such that $|\phi_i(x_0,y_0)|>Ct_j^k$. However, by item (iii), it follows that $X \cap \{(x,y)\in \C^{d+n}; t_1\leq \|x\|\leq t_j\}$ is a compact subset and, thus, by Maximum Principle (see Chirka \cite{Chirka:1989}, page 72, Theorem), the maximum of $\phi_i$ restrict to $X \cap \{(x,y)\in \C^{d+n}; t_1\leq \|x\|\leq t_j\}$ is attained in 
$$
M_j:=X \cap (\{(x,y)\in \C^{d+n}; \|x\|= t_j\}\cup \{(x,y)\in \C^{d+n}; \|x\|= t_1\}).
$$
Since $|\phi_i(w)|\leq Ct_j^k$ for all $w\in M_j$ and $(x_0,y_0)\in X \cap \{(x,y)\in \C^{d+n}; t_1\leq \|x\|\leq t_j\}$ satisfies $|\phi_i(x_0,y_0)|>Ct_j^k$, we obtain a contradiction. Therefore, Claim \ref{inequality} holds true.
\end{proof}

% Let us assume $d=\dim X>1$.

By Claim \ref{inequality}, we obtain that $\pi:=P|_{X\setminus (B\times \C^n)}\colon X\setminus (B\times \C^n)\to \C^d\setminus B$ is a proper mapping, where $B=\overline{B_{t_1}^{2d}(0)}$ if $d>1$ or $B=\emptyset$ if $d= 1$. Then, $\pi$ is a ramified cover and the critical values set of $\pi$, $\Sigma\subset \C^d\setminus B$, is a codimension $\geq 1$ complex analytic subset of the space $\C^d$. In particular, $\pi$ restrict to $X\setminus ((B\times \C^n)\cup \pi^{-1}(\Sigma))$ is a local diffeomorphism. Thus, $\varphi|_{\C^d\setminus (B\cup \Sigma) }\colon \C^d\setminus (B\cup \Sigma) \to X\setminus ((B\times \C^n)\cup \pi^{-1}(\Sigma))$ is a complex analytic mapping. We know that $\Sigma$ is a removable set, then $\varphi|_{\C^d\setminus (B\cup \Sigma) }$ can be extended to a complex analytic mapping $F=(f_1,...,f_n)\colon \C^d\setminus B\to X\setminus (B\times \C^n)$, since by hypotheses, $\varphi$ is locally bounded. If $d=\dim X>1$, then $F$ can be extended to a complex analytic mapping $\tilde F=(\tilde f_1,...,\tilde f_n)\colon \C^d \to \C^n$ and if $d=1$ then $B=\emptyset$ and, thus, $F$ is already defined in the whole $\C^d$ and, in this case, we set $\tilde F:=F$.

Since $\tilde F$ is bounded in $B$, by Claim \ref{inequality}, there exists $M>0$ such that each $\tilde f_i$ satisfies $|\tilde f_i(x)|\leq M t_j^k$ for all $x\in B^d_{t_j}(0)$ and for all $j$.

\begin{claim}\label{claim:Liouville_version}
$\tilde F$ is an polynomial mapping of degree at most $k$.
\end{claim}
\begin{proof}[Proof of Claim \ref{claim:Liouville_version}]
Fixed $i$, let $f=\tilde f_i$.
For each $x\in \mathbb{C}^d$, let $j_0\in \mathbb{N}$ such that $x\in B_{t_j}(0)$ for all $j\geq j_0$. Since $f$ satisfies $\|f\|_{L^{\infty}(B_{t_j}(0))}\leq M t_j^k$ for all $j\geq j_0$, by Cauchy's estimate, for each multi-index $\alpha$, we have
\begin{eqnarray*}
\left|D^{\alpha} f(x)\right|&\leq &\frac{\alpha !}{(2\pi i)^d }  \frac{\|f\|_{L^{\infty}(B_{t_j-\|x\|}(x))}}{(t_j-\|x\|)^{|\alpha|}} \\
&\leq &\frac{\alpha !}{(2\pi i)^d } \frac{\|f\|_{L^{\infty}(B_{t_j}(0))}}{(t_j-\|x\|)^{|\alpha|}} \\
&\leq &\frac{\alpha !}{(2\pi i)^d }\frac{Mt_j^k}{(t_j-\|x\|)^{|\alpha|}} 
\end{eqnarray*}
for all $j\geq j_0$, which implies $\left|D^{\alpha} f(x)\right|=0$, for any $\alpha$ satisfying $|\alpha|>k$.
Therefore $f$ is a polynomial of degree at most $k$. 
Thus, $\tilde F$ is a polynomial mapping of degree at most $k$.% and, in particular, $\varphi=\tilde F|_{\C^\setminus K}$.
\end{proof}

Thus, $Graph (\tilde F)$ is a complex algebraic set with degree at most $k$.
Since $X\setminus (B\times \C^n) \subset Graph (\tilde F)$ and $Graph (\tilde F)$ is locally irreducible, we have that $Graph(\tilde F)\subset X$. Moreover, since $\dim X=\dim Graph(\tilde F)$, if $X$ is an irreducible analytic set then $X=Graph(\tilde F)$ and, in this case, if we have also $k=1$ then $X$ is an affine linear subspace of $\C^{d+n}$.
% When $\dim X=1$, we consider $\tilde X=X\times \mathbb{C}$. Thus, $\tilde X$ satisfies the hypotheses of the proposition, then it is an affine linear subspace of $\C^{n+1}$, which implies that $X$ is an affine linear subspace of $\C^n$.
\end{proof}

\begin{remark}
Claim \ref{claim:Liouville_version} of the proof of Theorem  \ref{gen_Moser_one} holds true also for a harmonic function, i.e., if  $u\colon \mathbb{R}^m\to \mathbb{R}$ is a harmonic function that satisfies $|u(x)|\leq C t_j^k$ for all $x\in B^m_{t_j}(0)$ and for all $j$ and some $C\geq 0$ and $k\in \mathbb{N}$ then $u$ is a polynomial of degree at most $k$.
\end{remark}

\begin{remark}\label{rem:gen_Moser_one_sharp}
Theorem \ref{gen_Moser_one} does not hold true in general, if we remove one of the conditions (i)-(iii). In fact, consider the sequence $\{t_j\}$ given by $t_j=j+1$ for all $j$. Let $X=\{(x,y)\in \C^2;xy=1\}\cup\{(0,0)\}$ and $\varphi\colon \C\to \C$ be the function given by 
$$
\varphi(x)=\left\{\begin{array}{ll}
            1/x,& \mbox{ if }x\not=0\\
            0,& \mbox{ if }x=0.
           \end{array}\right.
$$
Then, $X=Graph(\varphi)$ and $|\varphi(x)|\leq |x|$ for all $x\in \C\setminus B_1^2(0)$, which implies that $\varphi$ satisfies (i) and (ii), but $X$ does not contain a line. Thus, we cannot remove condition (iii). By taking $\varphi|_{\C\setminus \overline{B_1^2(0)}}$, we see that condition (i) also cannot be removed. The set $Y=\{(x,y)\in \C^2;y=e^x\}$ clearly satisfies (i) and (iii), which shows that we cannot remove condition (ii). 
\end{remark}
\begin{remark}\label{rem:no_real_gen_Moser_one}
As it was already remarked in the introduction, Theorem \ref{gen_Moser_one} does not hold true for closed sets which minimizes volume (stable varieties), see Example \ref{example:LO}. 
% In order to know, in \cite[Theorem 7.1]{LawsonO:1977} is shown that there is a connected closed semialgebraic set $X\subset \mathbb R^{n+k}$ which minimizes volume and is a graph of a Lipschitz function $u\colon \mathbb{R}^{n}\to \mathbb{R}^k$. Moreover, $X$ is not an affine linear subspace.
\end{remark}

% 
% As it was already remarked in Remarks \ref{rem:gen_Moser_one_sharp} and \ref{rem:no_real_gen_Moser_one}, Theorem \ref{gen_Moser_one} does not hold true in general, if we remove one of the conditions (i)-(iii) and it does not hold true for closed sets which minimizes volume (stable varieties).

However, we have the following complex version of the result proved by Moser in \cite{Moser:1961}.
\begin{corollary}\label{gen_Moser_two}
Let $X\subset\C^{d+n}$ be a complex analytic subset. If $X$ is a graph of a Lipschitz mapping $\varphi\colon \C^d\to \C^n$, then $X$ is an affine linear subspace of $\C^{n+d}$.
\end{corollary}

% \begin{proof}
% The case $d>1$ follows from Theorem \ref{gen_Moser_one}.
% 
% \end{proof}

% We finish this paper with the following conjecture.
% 
% \begin{conjecture}
% Let $X\subset\C^n$ be an  entire pure $d$-dimensional complex analytic subset. Suppose that $X$ has a tangent cone at infinity which is an affine linear subspace of $\C^{n}$. If $X$ is Lipschitz regular at infinity, then $X$ is an affine linear subspace of $\C^n$.
% \end{conjecture}

% We obtain also the following consequences
% \begin{corollary}
% Let $X\subset\C^n$ be a pure $d$-dimensional closed analytic subset. Suppose that $X$ has a  tangent cone at infinity that is a complex cone. If $X$ is Lipschitz regular at infinity, then $X$ is an affine linear subspace of $\C^n$.
% \end{corollary}
% \begin{proof}
%  We are going to prove that $C_{\infty}(X)$ is a complex cone. In order to do that, it is enough to prove that for any $x\in C_{\infty}(X)$ then $\lambda x\in C_{\infty}(X)$ for all $\lambda \in \C$. So, let $x\in C_{\infty}(X)$.
% \end{proof}

\bigskip

% \noindent{\bf Acknowledgements}. 
% The author would like to thank Lev Birbrair, Alexandre Fernandes and Eur\'ipedes da Silva for their interest about this work.

% Non-BibTeX users please use

\end{document}